\documentclass[10pt]{article}
\usepackage{graphicx}
\usepackage{color}
\usepackage{amssymb}
\usepackage{amsmath}

\usepackage[utf8]{inputenc}
\usepackage{amsmath,amsfonts,bm,amsthm}
\usepackage{color}
\usepackage{stmaryrd}
\usepackage{authblk}

\usepackage{graphicx}
\usepackage{bm}
\usepackage{float}
\usepackage{hyperref}

\newcommand{\ii}{{\rm i}}
\newcommand{\dd}{{\rm d}}

\newtheorem{thm}{Theorem}[section] 
\newtheorem{lem}[thm]{Lemma}

\newtheorem{ex}[thm]{Example}
\newtheorem{remark}[thm]{Remark}

\title{Preconditioned discontinuous Galerkin method and convection-diffusion-reaction problems with guaranteed bounds to resulting spectra}
\author[1,3]{Liya Gaynutdinova}
\author[2]{Martin Ladeck\' y}
\author[1]{Ivana Pultarov\' a}
\author[1]{Miloslav Vlas\' ak}
\author[3]{Jan Zeman}
\date{}
\affil[1]{Department of Mathematics, Faculty of Civil Engineering, CTU in Prague}
\affil[2]{Department of Microsystems Engineering, University of Freiburg}
\affil[3]{Department of Mechanics, Faculty of Civil Engineering, CTU in Prague}

\begin{document}

\maketitle

 {\bf Abstract.} This paper focuses on the design, analysis and implementation of a new preconditioning concept for linear second order partial differential equations, 
 including the convection-diffusion-reaction problems
 discretized by Galerkin or discontinuous Galerkin methods.
  We expand on the approach introduced by Gergelits et 
al.~and adapt it to the more general settings,
 assuming that 
 both the original and preconditioning matrices are
 composed of sparse matrices of very low ranks, representing  local contributions to the global 
 matrices. When applied to a symmetric problem, the method provides bounds to all individual 
 eigenvalues of the preconditioned matrix. 
 We show that this preconditioning strategy works
 not only for Galerkin discretization, but also for the discontinuous Galerkin discretization, where local contributions are associated with  individual edges 
 of the triangulation. 
 In the case of non-symmetric problems,
 the method yields guaranteed bounds to real and imaginary parts of the resulting eigenvalues.
 We include some numerical experiments illustrating the method and its
 implementation, showcasing its effectiveness for the two 
 variants of discretized (convection-)diffusion-reaction
 problems.



\section{Introduction}

We introduce a preconditioning and eigenvalue estimation 
approach for solving systems of linear equations that stem from two distinct problems:
(a)  discretization of diffusion-reaction problems by discontinuous Galerkin (DG) method 
and (b) discretization of
convection-diffusion-reaction problems
by conforming Galerkin method.
While the former problem results in a system of linear equations with a real symmetric matrix, the 
latter may yield a non-symmetric matrix.
The main purpose of this paper
is to show that the special concept of preconditioning
introduced recently by Gergelits et al.~in~\cite{G}
and further developed in~\cite{GNS2020,LadeckyPZ2020,PuLa2021}
is applicable to an even broader area of 
second order partial differential equations (PDE) and
discretization schemes.
This view of preconditioning was motivated by paper by Nielsen et al.~\cite{NielsenTH2009}, where the inverse of Laplacian was
used as preconditioner to the equation $-\nabla\cdot(k\nabla u)=f$
with a continuous scalar function $k$. It was shown that the 
rank of $k$ is involved in the spectrum of the 
preconditioned operator.
Gergelits et al.~transferred this idea to the
discretized forms of operators~\cite{G} and also refined it by new theoretical results~\cite{GergelitsNS2022}.
In~\cite{GNS2020} this approach was generalized to 
problems with tensor material coefficients.
In~\cite{LadeckyPZ2020,PuLa2021} 
the algorithm was modified and applied to the solution of linear elasticity equation and to the discretization with the finite difference method and the stochastic Galerkin method.
This preconditioning strategy was successfully employed in~\cite{LadeckyLFPPJZ2023, LeuteLFJPZJP2022} to address problems arising from image-based homogenization of heterogeneous materials. Additionally, we highlighted its close connection with a class of fast homogenization algorithms proposed by Moulinec and Suquet~\cite{Moulinec1994,Moulinec1998} and widely used in the computational mechanics of materials community; see recent reviews~\cite{Schneider2021,Lucarini2022,Gierden2022} for additional details.

Let us note that other preconditioning methods use the idea of inspecting the preconditioned problem locally in order to obtain the 
whole picture of the spectrum.
The so-called strengthened 
Cauchy-Schwarz-Bunyakowski constant used in 
algebraic multilevel preconditioning can be obtained
from the local properties of the 
(possibly nonconforming) discretization and the data,
see e.g.~\cite{A,BMN,EijkhoutV}.
The resulting spectrum of a preconditioned matrix can be also estimated by 
the local Fourier analysis, which uses the so-called symbols to describe the asymptotic spectral distribution and  effectivity of some preconditioners (such as multigrid), see e.g.~\cite{Donatelli2016}.
Similar approach can be found in~\cite{Hemker2004},
where the local Fourier analysis is used for 
a block-Toeplitz matrix arising from a block relaxation multigrid method for the interior penalty DG method.

The basic assumption of our method of preconditioning and eigenvalue estimation is that the matrix of the 
discretized problem and the preconditioning matrix are obtained as sums of the same number of sparse matrices 
of a very low rank, and that corresponding pairs of matrices 
(of the problem and of the preconditioner) 
have the same or similar kernels. 
This assumption is naturally satisfied for
finite element method (FEM),
finite difference or stochastic Galerkin method
applied to elliptic problems,
for example; see~\cite{PuLa2021} for more details.
Our goal is to show how this basic concept 
extends to the two settings not considered 
for this type of preconditioning before.

The first problem addressed here is the DG method for solving an elliptic PDE yielding a symmetric system matrix. 
More specifically, we consider the symmetric interior penalty
DG (SIPG) method, which allows preserving the symmetry of the continuous problem. SIPG method brings the original idea of penalizing irregular boundary conditions, see e.g.~PDE analysis~\cite{Lions} or applications to FEM~\cite{Babuska}, to all inter-element connections, cf.~\cite{D-D}. 
The DG approach in general weakens the inter-element connections in comparison with the more traditional FEM, which gives it a number of advantages, e.g.~easier 
implementation of adaptive methods. Moreover, the increased number of degrees of freedom (DOFs) along element boundaries may be used for the stabilization, which results in more robust schemes. On the other hand, DG is often considered as rather computationally expensive in comparison with FEM due to the increased number of DOFs, especially for lower degree polynomial approximations. For more complete overview and analysis of DG, see the seminal paper \cite{Brezzi2002} or the books~\cite{D-E,DolejsiFeistauer2015} and the references cited therein.

The second problem involves using conforming Galerkin discretization to solve a second order partial differential equation that contains a convective term. This term makes the system of linear equations non-symmetric.
Although DG is commonly used to solve problems that involve convection (such as Navier-Stokes equations~\cite{DolejsiFeistauer2015}), since DG can be used as a stabilization technique for convective terms, we specifically address the use of conforming Galerkin discretization and DG separately. We do this to highlight the unique contributions that each technique brings to our preconditioning and spectral estimation method.
In the DG scheme we do not associate the local contribution matrices with integrals over individual elements,
as it was done in~\cite{G,LadeckyPZ2020,PuLa2021}. 
Instead, we associate them with the contributions given by integrals
over particular edges, where the amount of element-wise integrals
corresponding to neighboring elements is added.
In Galerkin discretization of a convective term, we deal with symmetric and 
skew-symmetric parts of the system matrix separately
and estimate the real and imaginary parts of the
resulting spectrum.

Probably the most popular preconditioners 
for Galerkin and DG methods are various types of multilevel and multigrid methods.
Multigrid can be used for DG for self-adjoint problems~\cite{Dobrev2006,Hartmann2009,Kronbichler2018},
for a non-conforming mesh~\cite{Brix2008,DeDiosZiatanov2009}
with a condition number estimation using the stable splitting
technique. For a higher order DG discretization 
$p$-multigrid method can be used~\cite{ThieleR2022}.
Additive Schwarz method was presented e.g.~in~\cite{LorcaKanschat2021}.
See, e.g.~\cite{Bastian2019} for multigrid for convection dominated problems.
It is important to emphasize that our suggested type of preconditioning can be competitive with other preconditioning methods, such as multigrid, only in some particular cases.
Specifically, in a case when the solution
of a system with a preconditioning matrix is very cheap; or, if the fast discrete Fourier transform can be employed; or, if it is worth to factorize and store 
the inverse of the preconditioner, e.g., in the case when many 
similar problems have to be solved and preconditioned with
the same matrix, which can be used in time dependent 
problems~\cite{Francolini2020}.

Additional benefit of a preconditioner with 
guaranteed spectral bounds on the system matrix is 
that it can be used for a~posteriori error estimates. 
Let us assume the problem $\mathsf{Ax=b}$ with 
the exact solution $\mathsf x$ and some approximate solution $\widetilde{\mathsf{x}}$. Usually, the most desired and 
physically appropriate norm of the error 
$\mathsf{e}=\mathsf{x}-\widetilde{\mathsf{x}}$
is the energy norm $\Vert \mathsf{e}\Vert_{\mathsf{A}}=
\sqrt{\mathsf{e}^T\mathsf{Ae}}$.
However, the residual $\mathsf{r=b-A}\widetilde{\mathsf{x}}=\mathsf{A}(\mathsf{x}-\widetilde{\mathsf{x}})$ is usually only accessible during the computation. Assuming 
$\mathsf{A}$ and $\mathsf{P}$ symmetric and positive definite (s.p.d.) and using some guaranteed bounds to the spectrum
$\sigma( \mathsf{P}^{-1}\mathsf{A})\in [c_1,c_2]$, $c_1>0$,
we obtain 
\begin{equation}\nonumber
\frac{1}{c_2}\,{\mathsf{r}^T\mathsf{P^{-1}r}}\le 
\Vert \mathsf{e}\Vert^2_{\mathsf{A}}=
{\mathsf{r}^T\mathsf{A}^{-1}\mathsf{r}}\le
\frac{1}{c_1}\,{\mathsf{r}^T\mathsf{P^{-1}r}}.
\end{equation}
Another strategy of fully computable a~posteriori error estimates can be found e.g.~in~\cite{AinsworthFu2028}.

The paper is organized in the following way. The next section 
contains theoretical basics of our method.
We recall the main proof of lower and upper bounds to all
eigenvalues of a preconditioned problem
when both matrices are s.p.d.~and are obtained as sums 
of sparse and low
rank matrices. Then we derive bounds for the spectrum of 
preconditioned non-symmetric matrices. Section~\ref{secDG}  
focuses on the DG method for 
a diffusion-reaction equations. Section~\ref{secG} focuses
on convection-diffusion-reaction equations.
In both sections we present algorithms and numerical examples.
A short discussion and some open problems conclude the paper.

Linear algebra objects, vectors and matrices, are denoted by
the sans serif font, e.g.~$\mathsf v$, $\mathsf A$. 
We do not distinguish
between scalar, vector or tensor variables and functions, e.g.~$x$, $u$, $a$; the meaning is always 
clear from the context.
For descriptive subscripts and superscripts like $\circ\,_{\text{max}}$, $\circ\,_{\text{loc}}$, $\circ\,_{\text{dof}}$, we use the roman font.
We abuse the notation when by the condition number $\kappa(\mathsf{P^{-1}A})$ of the matrix $\mathsf{P^{-1}A}$
we mean the ratio of maximum and minimum eigenvalues 
of $\mathsf{P^{-1}A}$. In fact, by the condition number
we mean $\kappa(\mathsf{P^{-\frac{1}{2}}AP^{-\frac{1}{2}}})$.
The abbreviations used in the following 
are: PDE (partial differential equation),
DG (discontinuous Galekin method),
SIPG (symmetric interior penalty DG),
CG (conjugate gradient method), GMRES (generalized minimal residual method),
FEM (finite element method), DOF (degree of freedom), 
s.p.d.~(symmetric positive definite),
and s.p.s.d.~(symmetric positive semi-definite).

\section{Eigenvalue bounds for preconditioned 
symmetric and non-symmetric matrices}

In this section we present a new general method which computes 
spectral bounds for a real valued matrix $\mathsf{A}+\mathsf{B}$, $\mathsf{A}$ symmetric and 
$\mathsf{B}$ skew-symmetric,
preconditioned by an s.p.d.~matrix $\mathsf{P}$.
We assume that all matrices $\mathsf{A},\mathsf{B},\mathsf{P}\in\mathbb{R}^{N\times N}$ are obtained as sums of 
matrices $\mathsf{A}_k,\mathsf{B}_k,\mathsf{P}_k\in\mathbb{R}^{N\times N}$, $k=1,\dots,N_{\rm loc}$, such that their kernels fulfill
\begin{equation}\nonumber
{\rm Ker}(\mathsf{A}_k)\subseteq {\rm Ker}(\mathsf{P}_k),\quad
{\rm Ker}(\mathsf{P}_k)\subseteq {\rm Ker}(\mathsf{B}_k),\quad k=1,\dots,N_{\rm loc}.
\end{equation}
The matrices $\mathsf{A}_k,\mathsf{B}_k,\mathsf{P}_k$
can represent in some sense local (with respect to the 
physical solution domain) contributions to the whole system.
Such a setting can be found, for example,
in composing matrices of FEM or of a finite difference or finite volume methods.
Some applications are shown in~\cite{G,LadeckyPZ2020,PuLa2021}, 
new applications are introduced in Sections~\ref{secDG} and~\ref{secG}.

Let us denote the spectrum of a matrix $\mathsf{M}$ by $\sigma(\mathsf{M})$.
It is well known that any real symmetric matrix $\mathsf{A}$ has only real eigenvalues. Then let us denote by $\lambda_{\rm min}(\mathsf{A})$
and $\lambda_{\rm max}(\mathsf{A})$ the minimal and maximal eigenvalues of $\mathsf A$.
On the other hand, all eigenvalues of a skew-symmetric matrix $\mathsf{B}$ have the real parts equal to zero.
Then let us denote by 
$\lambda_{\rm im,max}(\mathsf{B})$ the maximal imaginary part of the eigenvalues of a skew-symmetric matrix $\mathsf{B}$.
Notice that their minimal imaginary part is $-\lambda_{\rm im,max}(\mathsf{B})$.
In the following, we will use the notation
$\mathsf{P}^{-1}$ for the inverse of an s.p.d.~matrix 
$\mathsf P$ or for its pseudoinverse if $\mathsf{P}$ is singular; see e.g.~\cite{Golub}.

In the next lemma we recall some well known results.  Since they are important for the subsequent derivations, we also present the proof.
\begin{lem}\label{lem1}
Let $\mathsf{A}$ be symmetric and $\mathsf B$ skew-symmetric matrices, 
$\mathsf{A},\mathsf{B}\in\mathbb{R}^{N\times N}$. Let $\mu+\ii \xi$ be an eigenvalue of 
$\mathsf{A}+\mathsf{B}$, $\mu,\xi\in\mathbb{R}$. Then 
\begin{equation}\nonumber
\mu\in[\lambda_{\rm min}(\mathsf{A}),\lambda_{\rm max}(\mathsf{A})],\quad
\xi\in[-\lambda_{\rm im,max}(\mathsf{B}),\lambda_{\rm im,max}(\mathsf{B})].
\end{equation}
\end{lem}
\begin{proof}
Since $\mathsf{B}$ is skew-symmetric, we have for all $\mathsf{y}\in\mathbb{R}^N$
that $\mathsf{y}^T\mathsf{By}=\mathsf{y}^T\mathsf{B}^T\mathsf{y}=-\mathsf{y}^T\mathsf{By}$ which yields $\mathsf{y}^T\mathsf{By}=0$.
Let $(\mathsf{A}+\mathsf{B})(\mathsf{u}+\ii \mathsf{v})=(\mu+\ii \xi)(\mathsf{u}+\ii \mathsf{v})$, where $\mu,\xi\in\mathbb{R}$
and $\mathsf{u},\mathsf{v}\in\mathbb{R}^N$. Then we have
\begin{eqnarray}
\mathsf{u}^T\mathsf{Au}+\mathsf{v}^T\mathsf{Av}&=&
\mu(\mathsf{u}^T\mathsf{u}+\mathsf{v}^T\mathsf{v})\nonumber\\
2\mathsf{u}^T\mathsf{Bv}=\mathsf{u}^T\mathsf{Bv}-\mathsf{v}^T\mathsf{Bu}&=&\xi(\mathsf{u}^T\mathsf{u}+\mathsf{v}^T\mathsf{v}).\nonumber
\end{eqnarray}
Then we obviously have
\begin{equation}\nonumber
\lambda_{\rm min}(\mathsf{A})=\min_{\mathsf{y}\in \mathbb{R}^N,\, \mathsf{y}\ne 0}\frac{\mathsf{y}^T\mathsf{Ay}}{\mathsf{y }^T\mathsf{ y}}
\le \mu=\frac{\mathsf{u}^T\mathsf{Au}+\mathsf{v}^T\mathsf{ Av}}{\mathsf{u}^T\mathsf{u}+\mathsf{v}^T\mathsf{v}}\le 
\max_{\mathsf{y}\in \mathbb{R}^N,\, \mathsf{y}\ne \mathsf{0}}\frac{\mathsf{y}^T\mathsf{Ay}}{\mathsf{y}^T\mathsf{y}}=\lambda_{\rm max}(\mathsf{A}).
\end{equation}
Assume $\xi\ne 0$. Then $\mathsf{u}\ne 0$ and $\mathsf{v}\ne 0$.
Let us normalize the real part of the eigenvector $\mathsf{u}+\ii \mathsf{v}$, i.e.~let
$\Vert \mathsf{u}\Vert=1$. For $\xi$ we have
\begin{equation}\nonumber
 \vert\xi\vert=\frac{2\vert \mathsf{u}^T\mathsf{Bv}\vert}{\mathsf{u}^T\mathsf{u}+\mathsf{v}^T\mathsf{v}}=\frac{2\vert \mathsf{u}^T\mathsf{Bv}\vert}{1+\mathsf{v}^T\mathsf{v}}\le 
 \frac{2\left\vert \mathsf{u}^T\mathsf{B}\frac{\mathsf{v}}{\Vert \mathsf{v}\Vert} \right\vert}{1+\frac{\mathsf{v}^T\mathsf{v}}{\Vert \mathsf{v}\Vert^2}}
 =\left\vert \mathsf{u}^T\mathsf{B}\frac{\mathsf{v}}{\Vert \mathsf{v}\Vert} \right\vert=
 \frac{\vert \mathsf{u}^T\mathsf{Bv}\vert}{\Vert \mathsf{v}\Vert}  .
\end{equation}
(The inequality follows from the observation that 
if $d>0$ then $f(x)=\frac{x}{1+d^2x^2}$,
$x\in\mathbb{R}$, attains its maximum for $x=1/d$.) 
Then 
\begin{equation}\nonumber
\xi^2
\le 
\max_{\mathsf{y,w}\in\mathbb{R}^N,\,\Vert \mathsf{y}\Vert=1,\,\Vert \mathsf{w}\Vert=1}( \mathsf{y}^T\mathsf{B}\mathsf{w})^2
\le 
\max_{\mathsf{w}\in\mathbb{R}^N,\,\Vert \mathsf{w}\Vert=1}
\mathsf{w}^T\mathsf{B}^T\mathsf{Bw}
=
\max_{\mathsf{w}\in\mathbb{R}^N,\,\Vert \mathsf{w}\Vert=1}
-\mathsf{w}^T\mathsf{B}^2\mathsf{w}\le 
\rho(\mathsf{B})^2,
\end{equation}
where $\rho(\mathsf{B})$ is the spectral radius of $\mathsf{B}$.
\end{proof}

Analogously to Lemma~\ref{lem1} we can obtain the next result.
\begin{lem}\label{lem2}
Let $\mathsf{A},\mathsf{P}$ be symmetric and $\mathsf{B}$ skew-symmetric matrices, 
$\mathsf{A},\mathsf{P},\mathsf{B}\in\mathbb{R}^{N\times N}$ and let $\mathsf{P}$ be positive definite. Let $\mu+\ii \xi$ be an eigenvalue of 
$\mathsf{P}^{-1}(\mathsf{A}+\mathsf{B})$, $\mu,\xi\in\mathbb{R}$. Then 
\begin{equation}\nonumber
\mu\in[\lambda_{\rm min}(\mathsf{P}^{-1}\mathsf{A}),\lambda_{\rm max}(\mathsf{P}^{-1}\mathsf{A})],\quad
\xi\in[-\lambda_{\rm im,max}(\mathsf{P}^{-1}\mathsf{B}),\lambda_{\rm im,max}(\mathsf{P}^{-1}\mathsf{B})].
\end{equation}
\end{lem}
\begin{proof}
Let us notice that $\sigma(\mathsf{P }^{-1}\mathsf{ A})=\sigma(\mathsf{ P}^{-\frac{1}{2}}\mathsf{ AP}^{-\frac{1}{2}})$. Since $\mathsf{ P}^{-\frac{1}{2}}\mathsf{ AP}^{-\frac{1}{2}}$ is 
symmetric and $\mathsf{P }^{-\frac{1}{2}}\mathsf{ BP}^{-\frac{1}{2}}$ is skew-symmetric,
we can apply Lemma~\ref{lem1} to prove the statement.
\end{proof}

\subsection{Symmetric matrices $\mathsf{A}$ and $\mathsf{P}$}\label{SymMat}

We now recall the method 
of obtaining the individual eigenvalue bounds 
for a preconditioned s.p.d.~matrix. 
Let $\mathsf{A},\mathsf{P}\in\mathbb{R}^{N\times N}$ be s.p.d.~matrices such that
\begin{equation}\nonumber
\mathsf{A}=\sum_{k=1}^{N_{\rm loc}}\mathsf{A}_k,\quad \mathsf{P}=\sum_{k=1}^{N_{\rm loc}}\mathsf{P}_k,
\end{equation}
where $\mathsf{A}_k, \mathsf{P}_k$, are s.p.s.d., 
and 
\begin{equation}\label{kerAP}
{\rm Ker}(\mathsf{A}_k)={\rm Ker}(\mathsf{P}_k),\quad k=1,\dots,{N_{\rm loc}}.
\end{equation}
Let $S_{{\mathsf A},j}$  and $S_{{\mathsf P},j}$ denote sets of indices of matrices $\mathsf{A}_k$ and $\mathsf{P}_k$, respectively, $k=1\dots,N_{\rm loc}$,
the $j$-th row of which are nonzero, i.e.
\begin{equation}\nonumber
S_{{\mathsf{A}},j}=\{k\in \{1,\dots,N_{\rm loc}\}; \; {\rm row\;}j\;{\rm of}\; \mathsf{A}_k\;{\rm is\; nonzero}\},\quad j=1,\dots,N,
\end{equation}
and similarly for $S_{{\mathsf{P}},j}$.
We can call these sets {\it patches} of indices $j$.
From~\eqref{kerAP} we obtain $S_{{\mathsf{P}},j}=   S_{{\mathsf{A}},j}$,
and thus we can denote
\begin{equation}\nonumber
  S_j:= S_{{\mathsf{A}},j}= S_{{\mathsf{P}},j}.
\end{equation}
Note that the sets $S_j$ can be overlapping.
For s.p.s.d.~matrices 
$\mathsf{P}_k,\mathsf{A}_k$ with identical kernels, let us use the notation $\lambda_{\rm ker,min}(\mathsf{P}_k^{-1}\mathsf{A}_k)$ 
and $\lambda_{\rm ker,max}(\mathsf{P}_k^{-1}\mathsf{A}_k)$
for the minimal and maximal, respectively,
generalized eigenvalues of $\mathsf{A}_k\mathsf{u}=\lambda\mathsf{P}_k\mathsf{u}$ restricted to vectors orthogonal to ${\rm Ker}(\mathsf{P}_k)$. 
Then $0<\lambda_{\rm ker,min}(\mathsf{P}_k^{-1}\mathsf{A}_k)\le\lambda_{\rm ker,max}(\mathsf{P}_k^{-1}\mathsf{A}_k) <\infty  $,
$k=1,\dots,N_{\rm loc}$.
Let us define vectors 
$\widetilde\gamma_{\rm min},\widetilde\gamma_{\rm max}\in\mathbb{R}^N$
\begin{eqnarray}
(\widetilde{\gamma}_{\rm min})_j&=&\min\{\lambda_{\rm ker,min}(\mathsf{P}_k^{-1}\mathsf{A}_k),\, k\in S_j\}\nonumber\\
(\widetilde{\gamma}_{\rm max})_j&=&\max\{\lambda_{\rm ker,max}(\mathsf{P}_k^{-1}\mathsf{A}_k),\, k\in S_j\},\quad j=1,\dots,N,\nonumber
\end{eqnarray}
and sort both sequences in ascending order, yielding vectors
$\gamma_{\rm min},\gamma_{\rm max}\in\mathbb{R}^N$,
\begin{eqnarray}
    &&0<({\gamma}_{\rm min})_1\le ({\gamma}_{\rm min})_2\le\dots\le ({\gamma}_{\rm min})_N\nonumber\\
    &&0<({\gamma}_{\rm max})_1\le ({\gamma}_{\rm max})_2\le\dots\le ({\gamma}_{\rm max})_N.\nonumber
\end{eqnarray}

\begin{lem}\label{lem3}
Let $0<\lambda_1\le\lambda_2\le \dots\le\lambda_N$ be eigenvalues of 
$\mathsf{P}^{-1}\mathsf{A}$. Then 
\begin{equation}\nonumber
0<({\gamma}_{\rm min})_j\le\lambda_j\le ({\gamma}_{\rm max})_j,\quad j=1,\dots,N.
\end{equation}
\end{lem}
\begin{proof}
The proof can be found in~\cite{LadeckyPZ2020,PuLa2021}.
Since the proof provides an insight into the main idea of the method,
we present it here in a slightly different, shorter, and perhaps more accessible way.
Let us recall the Courant-Fischer min-max theorem~\cite{Golub}
\begin{equation}\nonumber
    \lambda_k=\min_{{\rm dim}\,V=k}\,\max_{\mathsf{v}\in V,\, \mathsf{v}\ne 0}
    \frac{\mathsf{v}^T\mathsf{Av}}{\mathsf{v}^T\mathsf{Pv}}
    =\max_{{\rm dim}\,V=N-k+1}\,\min_{\mathsf{v}\in V,\, \mathsf{v}\ne 0}
    \frac{\mathsf{v}^T\mathsf{Av}}{\mathsf{v}^T\mathsf{Pv}},
\end{equation}
where all $V$'s denote subspaces of $\mathbb{R}^N$.
Let us start with a lower bound to $\lambda_1$, the smallest eigenvalue of 
$\mathsf{P}^{-1}\mathsf{A}$. We have
\begin{equation}\nonumber
    \lambda_1=\min_{\mathsf{v}\ne 0}\frac{\mathsf{v}^T\mathsf{Av}}{\mathsf{v}^T\mathsf{Pv}}=\min_{ \mathsf{v}\ne \mathsf{0}}
 \frac{\sum_{k=1}^{N_{\rm loc}}\mathsf{v}^T\mathsf{A}_k\mathsf{v}}{\sum_{k=1}^{N_{\rm loc}}\mathsf{v}^T\mathsf{P}_k\mathsf{v}}
 \ge 
 \min_{k=1,\dots,N_{\rm loc}} \lambda_{\rm ker,min}(\mathsf{P}_k^{-1}\mathsf{A}_k)=: (\widetilde{\gamma}_{\rm min})_{j_1}
 = ({\gamma}_{\rm min})_1.
    \end{equation}
 Let us proceed with $\lambda_2$. We have
 \begin{equation}\nonumber
 \lambda_2=  \max_{{\rm dim}\,V=N-1}\min_{\mathsf{v}\in V,\, \mathsf{v}\ne 0}
    \frac{\mathsf{v}^T\mathsf{Av}}{\mathsf{v}^T\mathsf{Pv}} 
    \ge      
   \min_{\mathsf{v}\in V,\, \mathsf{v}\ne 0,\, \mathsf{v}_{j_1}=0}
    \frac{\sum_{k=1}^{N_{\rm loc}}\mathsf{v}^T\mathsf{A}_k\mathsf{v}}{\sum_{k=1}^{N_{\rm loc}}\mathsf{v}^T\mathsf{P}_k\mathsf{v}}\ge  
   \min_{j\ne j_1} (\widetilde{\gamma}_{\rm min})_j=:(\widetilde\gamma_{\rm min})_{j_2}\ge ({\gamma}_{\rm min})_2.
 \end{equation}
   For $\lambda_3$ we have
 \begin{eqnarray}\nonumber
 \lambda_3&=&  \max_{{\rm dim}\,V=N-2}\min_{\mathsf{v}\in V,\, \mathsf{v}\ne 0}
    \frac{\mathsf{v}^T\mathsf{Av}}{\mathsf{v}^T\mathsf{Pv}} 
    \ge      
   \min_{\mathsf{v}\in V,\, \mathsf{v}\ne 0,\, \mathsf{v}_{j_1}=0,\, \mathsf{v}_{j_2}=0}
    \frac{\sum_{k=1}^{N_{\rm loc}}\mathsf{v}^T\mathsf{A}_k\mathsf{v}}{\sum_{k=1}^{N_{\rm loc}}\mathsf{v}^T\mathsf{P}_k\mathsf{v}}\ge
   \min_{j\ne j_1,j_2} (\widetilde\gamma_{\rm min})_j=:(\widetilde{\gamma}_{\rm min})_{j_3}
   \nonumber\\
    &\ge  &
   ({\gamma}_{\rm min})_3.\nonumber
 \end{eqnarray}
 Continuing in this manner, we obtain   
 $\lambda_j\ge  ({\gamma}_{\rm min})_j$, $j=1,\dots,N$. Similarly,
  starting from $\lambda_N$ and using the opposite inequalities and the upper bounds, we obtain 
 $\lambda_j\le ({\gamma}_{\rm max})_j$, $j=1,\dots,N$.
 \end{proof}

\begin{remark}
If matrices $\mathsf A$ and $\mathsf P$ are singular and have the
same kernel, we can restrict our consideration onto the vectors 
orthogonal to the kernel and obtain the bounds to the positive eigenvalues of $\mathsf{P}^{-1}\mathsf{A}$ according to Lemma~\ref{lem3}. This can be the case, for example, if Neumann
or periodic boundary conditions are applied.
\end{remark}

\begin{remark}\label{remG}
In~\cite{G}, Gergelits et al.~proved using Hall's theorem (\cite{Bondy1976,Hall1935}) a result which could be 
expressed in our notation and under the assumptions of Lemma~\ref{lem3} as 
\begin{equation}\label{G_statement}
    \sigma(\mathsf{P}^{-1}\mathsf{A})\subset 
    \cup_{j=1}^N[(\widetilde{\gamma}_{\min})_j,(\widetilde{\gamma}_{\max})_j].
\end{equation}
Note that this result is stronger that that of Lemma~\ref{lem3},
because in addition to~\eqref{G_statement}, after
sorting the vectors $\widetilde\gamma_{\rm min}$ and
$\widetilde\gamma_{\rm max}$, we obtain the bounds
defined in Lemma~\ref{lem3}.
\end{remark}

\subsection{Symmetric matrices $\mathsf A$ and $\mathsf P$ and skew-symmetric matrix $\mathsf B$}\label{NonSymMat}

In this part we 
present a new method for obtaining the bounds to all
complex eigenvalues of a preconditioned real non-symmetric matrix. 
To this goal let $\mathsf{A},\mathsf{B},\mathsf{P}\in \mathbb{R}^{N\times N}$ and let $\mathsf{A}$ be symmetric, $\mathsf{P}$  s.p.d., and $\mathsf{B}$ skew-symmetric, and let
\begin{equation}\nonumber
\mathsf{A}=\sum_{k=1}^{N_{\rm loc}}\mathsf{A}_k,\quad \mathsf{B}=\sum_{k=1}^{N_{\rm loc}}\mathsf{B}_k, \quad \mathsf{P}=\sum_{k=1}^{N_{\rm loc}}\mathsf{P}_k.
\end{equation}
We assume that the matrices $\mathsf{P}_k$
are s.p.s.d., $\mathsf{A}_k$ symmetric and $\mathsf{B}_k$
skew-symmetric, and their kernels fulfill
\begin{equation}\nonumber
{\rm Ker}(\mathsf{P}_k)\subseteq {\rm Ker}(\mathsf{A}_k),\quad
{\rm Ker}(\mathsf{P}_k)\subseteq {\rm Ker}(\mathsf{B}_k),\quad k=1,\dots,N_{\rm loc}.
\end{equation}
Let $S_{\mathsf{P},j}$ be sets of indices of matrices (patches) 
the $j$-th row of which are nonzero, i.e.
\begin{eqnarray}\nonumber
S_{\mathsf{P},j}&=&\{k\in \{1,\dots,N_{\rm loc}\}; \; {\rm row\;}j\;{\rm of}\; \mathsf{P}_k\;{\rm is\; nonzero}\}.\nonumber
\end{eqnarray}
Note that the sets $S_{\mathsf{P},j}$ can be overlapping.
For symmetric matrices 
$\mathsf{P}_k,\mathsf{A}_k$, let us denote by $\lambda_{\rm ker,min}(\mathsf{P}_k^{-1}\mathsf{A}_k)$ 
and $\lambda_{\rm ker,max}(\mathsf{P}_k^{-1}\mathsf{A}_k)$
 the minimal and maximal, respectively,
generalized eigenvalues of $\mathsf{A}_k\mathsf{u}=\lambda\mathsf{P}_k\mathsf{u}$ restricted to vectors orthogonal to ${\rm Ker}(\mathsf{P}_k)$. 
For skew-symmetric matrices 
$\mathsf{B}_k$, let us use the notation $\lambda_{\rm ker,im,max}(\mathsf{P}_k^{-1}\mathsf{B}_k)$ 
for the maximal imaginary part of 
generalized eigenvalues of $\mathsf{B}_k\mathsf{u}=\lambda\mathsf{P}_k\mathsf{u}$ restricted to vectors orthogonal to ${\rm Ker}(\mathsf{P}_k)$. 
Let us define for $j=1,\dots,N$ 
\begin{eqnarray}
(\alpha_{\rm min})_j&=&\min\{\lambda_{\rm ker,min}(\mathsf{P}_k^{-1}\mathsf{A}_k),\, k\in S_{\mathsf{P},j}\}\nonumber\\
(\alpha_{\rm max})_j&=&\max\{\lambda_{\rm ker,max}(\mathsf{P}_k^{-1}\mathsf{A}_k),\, k\in S_{\mathsf{P},j}\},\nonumber\\
(\beta_{\rm max})_j&=&\max\{\lambda_{\rm ker,im,max}(\mathsf{P}_k^{-1}\mathsf{B}_k),\, k\in S_{\mathsf{P},j}\}.\nonumber
\end{eqnarray}

\begin{lem}\label{lemNonsym1}
Let $\mu+\ii \xi$ be an eigenvalue of 
$\mathsf{P}^{-1}(\mathsf{A}+\mathsf{B})$. 
Then 
\begin{equation}\nonumber
\mu\in[\min_{j=1,\dots,N}\{(\alpha_{\rm min})_j\},
\max_{j=1,\dots,N}\{(\alpha_{\rm max})_j\}],\quad
\xi\in[-\max_{j=1,\dots,N}\{(\beta_{\rm max})_j\},
\max_{j=1,\dots,N}\{(\beta_{\rm \max})_j\}].
\end{equation}
\end{lem}
\begin{proof}
Since by Lemma~\ref{lem2} real parts of the eigenvalues of $\mathsf{P}^{-1}(\mathsf{A}+\mathsf{B})$
are bounded by the extremal eigenvalues of $\mathsf{P}^{-1}\mathsf{A}$
and the imaginary parts are bounded by extremal eigenvalues of
$\mathsf{P}^{-1}\mathsf{B}$, we can just estimate the spectra of
$\mathsf{P}^{-1}\mathsf{A}$ and $\mathsf{P}^{-1}\mathsf{B}$
separately.
The proof is similar to the proof of Lemma~\ref{lem3}
but we now do not have the ambition to estimate the individual 
eigenvalues. We need only to estimates 
the smallest and largest eigenvalues of $\mathsf{P}^{-1}\mathsf{A}$ and 
the largest imaginary parts of the eigenvalues of $\mathsf{P}^{-1}\mathsf{B}$. 
For any eigenvalue $\lambda\in\mathbb{R}$ 
 of $\mathsf{P}^{-1}\mathsf{A}$ we obviously have
\begin{equation}\nonumber
 \lambda\ge\min_{\mathsf{v}\ne 0}\frac{\mathsf{v}^T\mathsf{Av}}{\mathsf{v}^T\mathsf{Pv}}=\min_{ \mathsf{v}\ne \mathsf{0}}
 \frac{\sum_{k=1}^{N_{\rm loc}}\mathsf{v}^T\mathsf{A}_k\mathsf{v}}{\sum_{k=1}^{N_{\rm loc}}\mathsf{v}^T\mathsf{P}_k\mathsf{v}}
 \ge 
 \min_{j=1,\dots,N} (\alpha_{\rm min})_j.
 \end{equation}
Analogously, we obtain the upper bound
$\lambda\le \max_{j=1,\dots,N} (\alpha_{\rm max})_j$. For any
eigenvalue $\ii \nu$ of $\mathsf{P}^{-1}\mathsf{B}$,
$\nu\in\mathbb{R}$, we have
$\mathsf{B}(\mathsf{u}+\ii\mathsf{v})=\ii\nu\mathsf{P}(\mathsf{u}+\ii\mathsf{v})$ for some eigenvector $\mathsf{u}+\ii\mathsf{v}$,
$\mathsf{u,v}\in\mathbb{R}^N$, and thus
\begin{eqnarray}\nonumber
 \vert\nu\vert&=&\frac{2\vert\mathsf{u}^T\mathsf{Bv}\vert}{\mathsf{u}^T\mathsf{Pu}+\mathsf{v}^T\mathsf{Pv}} =
\frac{2\vert\sum_{k=1}^{N_{\rm loc}}\mathsf{u}^T\mathsf{B}_k\mathsf{v}\vert}{\sum_{k=1}^{N_{\rm loc}}(\mathsf{u}^T\mathsf{P}_k\mathsf{u}+
\mathsf{v}^T\mathsf{P}_k\mathsf{v})}\le
\max_{k=1,\dots,N_{\rm loc}}
\frac{2\vert\mathsf{u}^T\mathsf{B}_k\mathsf{v}\vert}{\mathsf{u}^T\mathsf{P}_k\mathsf{u}+
\mathsf{v}^T\mathsf{P}_k\mathsf{v}}\nonumber\\
&\le&\max_{k=1,\dots,N_{\rm loc}}\lambda_{\rm ker,im,max}
(\mathsf{P}_k^{-1}\mathsf{B}_k)\le
\max_{j=1,\dots,N}(\beta_{\rm max})_j.
\nonumber
 \end{eqnarray}
\end{proof}

\begin{remark}\label{rem_counter}{ 
In Remark~\ref{remG} we recall the result of~\cite{G} that the spectrum of a
symmetric preconditioned matrix is contained in a union 
of intervals defined by minimal and maximal eigenvalues 
of $\mathsf{P}_k^{-1}\mathsf{A}_k$ on every patch.
We might be interested whether the individual complex eigenvalues of $\mathsf{P}^{-1}(\mathsf{A}+\mathsf{B})$ can be analogously reliably bounded by rectangles in the complex plane defined by extreme real and imaginary  eigenvalues of 
$\mathsf{P}_k^{-1}\mathsf{A}_k$ and
$\mathsf{P}_k^{-1}\mathsf{B}_k$, respectively, 
over the patches. Unfortunately, the answer is negative, at least for this particular form of the statement. Let us introduce a counter-example. Let
\begin{equation}\nonumber\footnotesize
 \mathsf{A}_1+\mathsf{B}_1=\left(\begin{array}{cccc}
   10   & 12 &0 &0 \\ 
   -12   & 11 &0 &0 \\ 
     0 &  0 & 0& 0\\ 
      0 &0   & 0& 0
 \end{array}   \right),\;
 \mathsf{A}_2+\mathsf{B}_2=\left(\begin{array}{cccc}
 0 & 0  &0 &0 \\ 
 0 & 10   & 11 & 0 \\ 
0 &  -11   & 10 & 0 \\       
    0   & 0  & 0& 0
 \end{array}   \right),
 \;
 \mathsf{A}_3+\mathsf{B}_3=\left(\begin{array}{cccc}
 0 & 0  & 0&0 \\ 
 0 & 0  &0 & 0\\
0& 0 & 8   & 11   \\ 
0&0 &  -11   & 10   
 \end{array}   \right),
\end{equation}
and
\begin{equation}\nonumber\footnotesize
 \mathsf{P}_1=\left(\begin{array}{rrrr}
   1   & 0 &0 &0 \\ 
   0   & 1 &0 &0 \\ 
     0 &  0 & 0& 0\\ 
      0 &0   & 0& 0
 \end{array}   \right),\quad
 \mathsf{P}_2=\left(\begin{array}{rrrr}
 0 & 0  &0 &0 \\ 
 0 & 1   & 0 & 0 \\ 
0 &  0   & 1 & 0 \\       
    0   & 0  & 0& 0
 \end{array}   \right),
 \quad
 \mathsf{P}_3=\left(\begin{array}{rrrr}
 0 & 0  & 0&0 \\ 
 0 & 0  &0 & 0\\
0& 0 & 1   & 0   \\ 
0&0 &  0   & 1   
 \end{array}   \right).
\end{equation}
Then we obtain 
\begin{eqnarray}
    &&(\alpha_{\rm min})_1=(\alpha_{\rm min})_2=10,\quad
    (\alpha_{\rm min})_3=(\alpha_{\rm min})_4=8,\nonumber    \\
   &&(\alpha_{\rm max})_1=(\alpha_{\rm max})_2=11,\quad
   (\alpha_{\rm max})_3=(\alpha_{\rm max})_4=10,\nonumber\\
   &&(\beta_{\rm max})_1=   (\beta_{\rm max})_2=12,\quad
   (\beta_{\rm max})_3= (\beta_{\rm max})_4=11,\nonumber
\end{eqnarray}
and the eigenvalues of $\mathsf{P}^{-1}(\mathsf{A+B})$
\begin{equation}\nonumber
    \lambda_1\approx 9.881 + 11.322\ii,\quad
  \lambda_2\approx 9.881 - 11.322\ii,\quad
 \lambda_3\approx 9.869 + 5.827\ii,\quad 
 \lambda_4\approx 9.869 - 5.827\ii.
   \end{equation} 
   We can see that $ \lambda_1$ and $ \lambda_2$ do not fall into
   any rectangle $[(\alpha_{\rm min})_j,(\alpha_{\rm max})_j]
   \times   [-(\beta_{\rm max})_j\ii,(\beta_{\rm max})_j\ii]  $,
   $j=1,2,3,4$.
}\end{remark}

\section{Preconditioned discontinuous Galerkin method\\
 and diffusion-reaction problem}\label{secDG}

In this part we apply the method of preconditioning and 
evaluating the individual eigenvalue  bounds to the
diffusion-reaction problem with Dirichlet boundary condition which is discretized by the DG method~\cite{DolejsiFeistauer2015,ErnG}. 
We consider a polygonal domain $\Omega\subset{\mathbb R}^2$ divided 
into $N_{\tau}$ triangles $\tau_m$, $m=1,\dots,N_{\tau}$. We assume that the mesh is conforming, i.e.~the neighbouring elements share the entire edge, and that the elements $\tau_m$ are shape regular, i.e.~the ratios of the diameters of the inscribed and circumscribed circles are uniformly bounded. 
The set of interior edges is denoted by $\varepsilon^{\rm i}_k$, $k=1,\dots,N_{\varepsilon,\ii}$ and the set of boundary edges 
is denoted by $\varepsilon^{\rm b}_k$, $k=1,\dots,
N_{\varepsilon,\rm{b}}$. Moreover, we denote the unit normal on edge $\varepsilon^{\rm i}_k$ or $\varepsilon^{\rm b}_k$ as $n$. The directions of normal vectors are arbitrary but fixed for the interior edges and outward for the boundary edges.

The model problem reads 
\begin{equation}\label{model1}
-\nabla\cdot\left(a(x)\nabla u(x)\right)+c(x)u(x)=f(x),\;\; x\in \Omega,\qquad u(x)=0,\;\; x\in\partial\Omega,
\end{equation}
where $a:\Omega \to {\mathbb R}^{2\times 2}$ is essentially bounded
and uniformly s.p.d.~in $\Omega$,
$c\in L^\infty(\Omega)$ is nonnegative, and
$f\in L^2(\Omega)$. We define
the discontinuous finite element space
\begin{equation}\nonumber
V_{h}^s=\{v\in L^2(\Omega);\; v|_{\tau_m}\in P_s(\tau_m),\; m=1,\dots,N_\tau\}
\end{equation}
where $P_s(\tau_m)$ is the space of polynomials of total degree 
at most $s$ defined on $\tau_m$.
For a piece-wise continuous function $v$, we define one-sided values on the edges marked by L (left) and R (right),
respectively,
\begin{equation}
  v^{\rm L}(x) = \lim_{t \to 0+}v(x-tn),\qquad v^{\rm R}(x) = \lim_{t \to 0+}v(x+tn).
\end{equation}
With the aid of these one-sided values, we define the jump of a function on an interior and boundary edge
\begin{eqnarray}\nonumber
\llbracket v\rrbracket_{\varepsilon^{\rm i}_k} =  
(v^{\rm L}-v^{\rm R})n, \qquad 
\llbracket v\rrbracket_{\varepsilon_k^{\rm b}} =  
v^{\rm L}n, \nonumber
 \end{eqnarray}
respectively, and the average on an interior and boundary edge
\begin{eqnarray}\nonumber
\{\!\{ v\}\!\}_{\varepsilon^{\rm i}_k} =\frac{1}{2}\left(  
v^{\rm L}+v^{\rm R}\right), \qquad 
\{\!\{ v\}\!\}_{\varepsilon_k^{\rm b}} = v^{\rm L}, \nonumber
 \end{eqnarray}
 respectively.
We consider the biliner form on $L^2(\Omega)$
\begin{equation}\nonumber
(u,v)=\int_{\Omega}uv\,
{\rm d}x=\sum_{m=1}^{N_\tau}\int_{\tau_m}uv\,
{\rm d}x
\end{equation}
and define the bilinear forms on $V_h^s\times V_h^s$
\begin{eqnarray}
A(u,v)_h&=&\sum_{m=1}^{N_\tau}\int_{\tau_m}a\nabla u\cdot\nabla v\,
{\rm d}x, \nonumber\\
C(u,v)_h&=&\sum_{m=1}^{N_\tau}\int_{\tau_m}c\, u v\,
{\rm d}x,\nonumber\\
S(u,v)_h&=&\sum_{k=1}^{N_{\varepsilon,\rm i}}\int_{\varepsilon^{\rm i}_k}\{\!\{ a\nabla u\}\!\}\llbracket v \rrbracket\,{\rm d}s+\sum_{k=1}^{N_{\varepsilon,\rm b}}\int_{\varepsilon^{\rm b}_k}\{\!\{ a\nabla u\}\!\}\llbracket v \rrbracket{\rm d}s,\nonumber\\
J^\sigma(u,v)_h&=&\sum_{k=1}^{N_{\varepsilon,\rm i}}\int_{\varepsilon^{\rm i}_k}\sigma\llbracket u \rrbracket\llbracket v \rrbracket\,{\rm d}s+\sum_{k=1}^{N_{\varepsilon,\rm b}}\int_{\varepsilon^{\rm b}_k}\sigma\llbracket u \rrbracket\llbracket v \rrbracket{\rm d}s.\nonumber
\end{eqnarray}
Finally, we define the bilinear form
\begin{eqnarray}
\mathcal{A}(u,v)_h &=&A(u,v)_h-S(u,v)_h-S(v,u)_h+J^\sigma(u,v)_h+C(u,v)_h,
\nonumber
\end{eqnarray}
where a sufficiently large $\sigma>0$ at the penalization term $J^\sigma(.,.)_h$
guarantees positive definiteness of 
the bilinear form $\mathcal{A}(u,v)_h$ on $V_h^s\times V_h^s$. For the discussion of suitable choices of $\sigma$ see, e.g.~\cite{DolejsiFeistauer2015,Epshteyn}.
The discretized problem now reads to find $u_h\in V_h^s$ such that
\begin{equation}\nonumber
\mathcal{A}(u,v)_h=(f,v), \quad \forall v\in V_h^s.
\end{equation}

Let us assume that the FE basis functions are defined by
$N_{\rm dof}$ DOFs, typically  some nodal values 
in every element $\tau_m$, $m=1,\dots,N_\tau$.
In general, numbers of DOFs in individual elements may differ due to possibly
different approximation spaces on different elements.
The system of linear equations reads
\begin{equation}\label{Aub}
\mathsf{Au}=\mathsf{f}
\end{equation}
where $\mathsf{A}\in \mathbb{R}^{N_{\rm dof}\times N_{\rm dof}}$
and $\mathsf{f}\in \mathbb{R}^{N_{\rm dof}}$.
Denoting the FE basis functions by $\varphi_j$,
$j=1,\dots,N_{\rm dof}$, and from
$u=\sum_{j=1}^{N_{\rm dof}}\mathsf{u}_j\varphi_j$, we have 
\begin{equation}\nonumber
\mathsf{A}_{ij}=\mathcal{A}(\varphi_j,\varphi_i)_h,
\quad \mathsf{f}_i=(f,\varphi_i), \quad i,j=1,\dots, N_{\rm dof}.
\end{equation}
The preconditioning matrix $\mathsf P$ is obtained in the same way as $\mathsf A$ but for different data, namely $a_{\rm p},c_{\rm p},\sigma_{\rm p}$,
which are usually called {\it reference data}.
Usually, the reference data $a_{\rm p}$ and $c_{\rm p}$ are constant on $\Omega$
or can have a property which allow for easy 
solving a system of linear equations with the matrix $\mathsf P$.

Here the matrices $\mathsf{A}_k$ and $\mathsf{P}_k$ 
from Section~\ref{SymMat} are associated with an interior edge $\varepsilon_k^{\rm i}$ attached to two elements, e.g.,~$\tau_r$ and $\tau_s$; or with a boundary edge
$\varepsilon_k^{\rm b}$, where only one element is attached and
which can be formally expressed by $\tau_s=\emptyset$, 
\begin{eqnarray}\label{A_terms}
    (\mathsf{A}_k)_{ij}&=&\frac{1}{3}\int_{\tau_r\cup\tau_s}
    a\nabla \varphi_j\cdot\nabla\varphi_i\,\dd x+
  \frac{1}{3}\int_{\tau_r\cup\tau_s}
    c \,\varphi_j\varphi_i\,\dd x  \\
    &&- \int_{\varepsilon_k} 
    \{\!\{  a\nabla \varphi_i\}\!\} 
    \llbracket     \varphi_j     \rrbracket
    +
    \{\!\{ a\nabla \varphi_j\}\!\}\llbracket \varphi_i \rrbracket\,\dd s+\int_{\varepsilon_k}   \sigma 
 \llbracket \varphi_j \rrbracket\,\llbracket \varphi_i \rrbracket\,\dd s\nonumber
\end{eqnarray}
and
\begin{eqnarray}\label{P_terms}
    (\mathsf{P}_k)_{ij}&=&\frac{1}{3}\int_{\tau_r\cup\tau_s}
    a_{\rm p}\nabla \varphi_j\cdot\nabla\varphi_i\,\dd x+
  \frac{1}{3}\int_{\tau_r\cup\tau_s}
    c_{\rm p} \varphi_j\varphi_i\,\dd x  \\
    &&- \int_{\varepsilon_k} 
     \{\!\{ a_{\rm p}\nabla \varphi_i\}\!\}
    \llbracket \varphi_j\rrbracket
    +
    \{\!\{ a_{\rm p}\nabla \varphi_j\}\!\}\llbracket \varphi_i \rrbracket\,\dd s+\int_{\varepsilon_k}
 \sigma_{\rm p}\llbracket \varphi_j \rrbracket\,\llbracket \varphi_i \rrbracket\,\dd s.\nonumber
\end{eqnarray}
If we wish to apply Lemma~\ref{lem1}, then
$N_{\rm loc}$ is here the number of all edges,
$N_{\rm loc}=N_{\varepsilon,\rm i}+N_{\varepsilon,\rm b}$,
and $N$ is equal to the the number of DOFs,
$N=N_{\rm dof}$.
The factor $\frac{1}{3}$ is used 
in formulae~\eqref{A_terms} and~\eqref{P_terms}
because of the element integrals contributing to every edge integral
with the same amount.
In our setting, we use $s=1$ in $V^s_h$, i.e.~linear polynomials are used on every element and their nodal values are
chosen as DOFs. 
Then the matrices $\mathsf{A}_k$ and $\mathsf{P}_k$
have only $6\times 6$ non-zero elements, which correspond to 
$6$ DOFs on a pair of adjacent triangles which are connected by the edge $\varepsilon_k^{\rm i}$.
The first term in~\eqref{A_terms} produces two $3\times 3$ s.p.s.d.~matrices (or a $6\times 6$ block diagonal s.p.s.d.~matrix). The second term yields
two $3\times 3$ s.p.d.~matrices. The third term yields
a $6\times 6$ symmetric indefinite matrix. The fourth term is a penalty term 
yielding a $4\times 4$ s.p.s.d.~matrix.
If $\mathsf{A}_k$ and $\mathsf{P}_k$ correspond to a 
boundary edge, these respective matrices have
sizes $3\times 3$, $3\times 3$, $3\times 3$,
and $2\times 2$, respectively.

Let us now introduce the algorithm for 
constructing $\mathsf{A}$ and $\mathsf{P}$ and 
the bounds $\gamma_{\rm min},\gamma_{\rm max}\in\mathbb{R}^{N_{\rm dof}}$
to all individual eigenvalues of the preconditioned matrix $\mathsf{P}^{-1}\mathsf{A}$ according to Lemma~\ref{lem3}.

{\bf Algorithm~1.} (Constructing $\mathsf{A,P}\in \mathbb{R}^{N_{\rm dof}\times N_{\rm dof}}$ and
eigenvalue bounds $\gamma_{\rm min},\gamma_{\rm max}\in\mathbb{R}^{N_{\rm dof}}$.)
\begin{itemize}
    \item[1.]  
    Set vectors $\widetilde\gamma_{\rm min},\widetilde\gamma_{\rm max}\in\mathbb{R}^{N_{\rm dof}}$ as
    $(\widetilde\gamma_{\rm min})_j=\infty$ and $(\widetilde\gamma_{\rm max})_j=0$, $j=1,\dots,{N_{\rm dof}}$. 
    \item[2.]
    For every interior edge $\varepsilon_k^{\rm i}$, $k=1,\dots,N_{\varepsilon,{\rm i}}$: \\
    Set $\widetilde{\mathsf{A}}_k, \widetilde{\mathsf{P}}_k \in 
    \mathbb{R}^{6\times 6}$ according to~\eqref{A_terms}
    and~\eqref{P_terms}, respectively.\\
    Add matrices $\widetilde{\mathsf{A}}_k,\widetilde{\mathsf{P}}_k$ to elements at positions given by indices $j_1,\dots,j_6$ in $\mathsf{A}$ and $\mathsf{P}$, respectively.\\
    Compute generalized eigenvalues of 
    $\widetilde{\mathsf{A}}_k\mathsf{v}=\lambda\widetilde{\mathsf{P}}_k \mathsf{v}$\\
    and denote 
    $\gamma_1:=\lambda_{\rm ker,min}(\widetilde{\mathsf{P}}_k^{-1}\widetilde{\mathsf{A}}_k)$
    and     $\gamma_6:=\lambda_{\rm ker,max}(\widetilde{\mathsf{P}}_k^{-1}\widetilde{\mathsf{A}}_k)$.\\ 
    For $s=1,\dots,6$: \\
    $(\widetilde\gamma_{\rm min})_{j_s}:=\min\{\gamma_1, (\widetilde\gamma_{\rm min})_{j_s}\}$,\\
    $(\widetilde\gamma_{\rm max})_{j_s}:=\max\{\gamma_6, (\widetilde\gamma_{\rm max})_{j_s}\}$.
    \item[3.]  
    For every boundary edge $\varepsilon^{\rm b}_k$, $k=1,\dots,N_{\varepsilon,{\rm b}}$: \\
    Set $\widetilde{\mathsf{A}}_k,\widetilde{\mathsf{P}}_k \in 
    \mathbb{R}^{3\times 3}$ according to~\eqref{A_terms}
    and~\eqref{P_terms}, respectively, with only one adjacent element.\\
Add matrices $\widetilde{\mathsf{A}}_k,\widetilde{\mathsf{P}}_k$ to elements at positions given by indices $j_1,j_2,j_3$ in $\mathsf{A}$ and $\mathsf{P}$, respectively.\\
    Compute generalized eigenvalues of 
    $\widetilde{\mathsf{A}}_k\mathsf{v}=\lambda\widetilde{\mathsf{P}}_k \mathsf{v}$\\
    and denote 
    $\gamma_1:=\lambda_{\rm ker,min}(\widetilde{\mathsf{P}}_k^{-1}\widetilde{\mathsf{A}}_k)$
    and     $\gamma_3:=\lambda_{\rm ker,max}(\widetilde{\mathsf{P}}_k^{-1}\widetilde{\mathsf{A}}_k)$. \\ 
    For $s=1,2,3$: \\
    $(\widetilde\gamma_{\rm min})_{j_s}:=\min\{\gamma_1, (\widetilde\gamma_{\rm min})_{j_s}\}$\\
    $(\widetilde\gamma_{\rm max})_{j_s}:=\max\{\gamma_3, (\widetilde\gamma_{\rm max})_{j_s}\}$
    \item[4.]
    Sort the vectors $\widetilde\gamma_{\rm min},\widetilde\gamma_{\rm max}\in\mathbb{R}^{N_{\rm dof}}$ 
    in increasing order to obtain  
    $\gamma_{\rm min},\gamma_{\rm max}\in\mathbb{R}^{N_{\rm dof}}$.
\end{itemize}

We introduced Algorithm~1 only for the case of 
homogeneous Dirichlet boundary conditions and 
triangular elements with element-wise linear approximation polynomials in order not to involve too many parameters 
in the exposition. However, the algorithm can be easily modified to any approximation space and any
boundary conditions.

\begin{ex}\label{example1}{\rm
In the following numerical example we show
preconditioning of the diffusion-reaction problem~\eqref{model1} where $\Omega=(0,1)\times (0,1)$,
\begin{equation}\nonumber
    a(x) = \left(\begin{array}{cc}
    3.01+3\sin(x_1x_2\pi)&0\\
    0&1.01+\sin(x_1x_2\pi)
    \end{array}\right),
\end{equation}
$c=1$, $c_{\rm p}=1$, $f=10$.
The penalty coefficient $\sigma$ and $\sigma_{\rm p}$, respectively, are chosen according to~\cite{Epshteyn},
which in our setting results in
\begin{equation}\nonumber
       \sigma= \frac{3\,c_\sigma}{\vert \varepsilon_k^{\rm i}\vert}\left(       
       \frac{\lambda_{\rm max}(a^{\rm L}(x))^2}
       {\lambda_{\rm min}(a^{\rm L}(x))}+
       \frac{\lambda_{\rm max}(a^{\rm R}(x))^2}
       {\lambda_{\rm min}(a^{\rm R}(x))}\right),\qquad
    \sigma_{\rm p}= \frac{3\,c_\sigma}{\vert \varepsilon^{\rm i}_k\vert}
    \left(  \frac{\lambda_{\rm max}(a_{{\rm p}}^{\rm L}(x))^2}
       {\lambda_{\rm min}(a_{{\rm p}}^{\rm L}(x))}+
       \frac{\lambda_{\rm max}(a_{{\rm p}}^{\rm R}(x))^2}
       {\lambda_{\rm min}(a_{{\rm p}}^{\rm R}(x))}\right),
\end{equation}
for interior edges, and 
\begin{equation}\nonumber
       \sigma= \frac{6\,c_\sigma}{\vert \varepsilon^{\rm b}_k\vert}      
       \frac{\lambda_{\rm max}(a^{\rm L}(x))^2}
       {\lambda_{\rm min}(a^{\rm L}(x))},\qquad
    \sigma_{\rm p}= \frac{6\,c_\sigma}{\vert \varepsilon^{\rm b}_k\vert}
      \frac{\lambda_{\rm max}(a_{{\rm p}}^{\rm L}(x))^2}
       {\lambda_{\rm min}(a_{{\rm p}}^{\rm L}(x))},
\end{equation}
for boundary edges, respectively, where 
$\vert \varepsilon_k\vert$ is the length of the edge $\varepsilon_k$,
$a^{\rm L}$ and $a^{\rm R}$ are diffusion coefficients
of the problem on the two respective attached elements, and
similarly, $a_{{\rm p}}^{\rm L}$ and $a_{{\rm p}}^{\rm R}$ are diffusion coefficients of the preconditioning problem
on the respective attached elements. All of the coefficient functions
are considered constant on every element in order to preserve 
exactness of quadrature formulas.
Here $\lambda_{\rm min}(a)$ and $\lambda_{\rm max}(a)$
define the minimal and maximal eigenvalues of $a$.
Constant $c_\sigma$ must be greater than one~\cite{Epshteyn}.
In our example, it is chosen as $c_\sigma=2$ or $20$.
In the following tables we can see condition numbers of 
(un)preconditioned (discontinuous) Galerkin 
method, the upper bounds to the 
condition numbers of the preconditioned matrices
obtained from Algorithms~1, and
numbers of iteration steps of CG to achieve reduction
of the residual by factor $10^{-6}$.
We use a uniform discretization with $N_1\times N_2$ subintervals.
First we show the results for the unpreconditioned
Galerkin (with continuous and piece-wise linear basis functions)
and DG methods:\\

\centerline{\begin{tabular}{c|cc|ccc|ccc}
$N_1=N_2$& $\kappa(\mathsf{A}_{\rm G})$&
 it CG  & $c_\sigma=2$ &$\kappa(\mathsf{A}_{\rm DG})$ &  it CG & $c_\sigma=20$ &$\kappa(\mathsf{A}_{\rm DG})$ &  it CG  \\
\hline
$10$ & 5.0e1 &  25   & & 4.6e3 &  213 & &4.2e4  & 367 \\
$20$ & 2.1e2 &  66   & & 1.9e4 & 519  & & 1.7e5 & 1144 \\ 
$30$ & 4.8e2 &   103  & & 4.2e4 & 990  & & 3.9e5 & 1997 \\
$40$ & 8.5e2 &  187   & & 7.6e4 &  1422 & & 7.0e5 & 3314 
\end{tabular}
 }

For two different preconditioning with data $a_{\rm p,1}$ or
$a_{\rm p,2}$ 
\begin{equation}\nonumber
 a_{\rm p,1}(x) = \left(\begin{array}{cc}
    1&0\\    0&1
    \end{array}\right),\qquad
    a_{\rm p,2}(x) = \left(\begin{array}{cc}
    3&0\\    0&1
    \end{array}\right),
\end{equation}
we obtain the following results of the Galerkin discretization
with continuous basis functions:\\

\centerline{\begin{tabular}{c|cccc|cccc}
$N_1=N_2$& $a_{\rm p,1}$ &$\kappa(\mathsf{P^{-1}_{\rm G}A}_{\rm G})$ & $\frac{(\gamma_{\rm max})_{N_{\rm dof}}}
{(\gamma_{\rm min})_{1}}$ 
& it PCG & $a_{\rm p,2}$& $\kappa(\mathsf{P^{-1}_{\rm G}A}_{\rm G})$ & $\frac{(\gamma_{\rm max})_{N_{\rm dof}}}
{(\gamma_{\rm min})_{1}}$ & it PCG 
 \\
\hline
$10$ & &  4.5 & 6.0  & 9 &  &  1.9  & 2.0 & 5   \\
$20$ & &  5.3 & 6.0  & 9  &  & 2.0   & 2.0 &  5 \\
$30$ & &  5.6 &  6.0 & 10 & & 2.0  & 2.0 &  5 \\
$40$ & & 5.7  &  6.0 & 10 &  &   2.0 & 2.0 &   5
\end{tabular}
}
and two sets of results for the DG discretization
for $c_\sigma=2$: \\

\centerline{\begin{tabular}{c|cccc|cccc}
$N_1=N_2$&  $a_{\rm p,1}$& $\kappa(\mathsf{P^{-1}_{\rm DG}A}_{\rm DG})$ & $\frac{(\gamma_{\rm max})_{N_{\rm dof}}}
{(\gamma_{\rm min})_{1}}$ & it PCG 
&  $a_{\rm p,2}$&$\kappa(\mathsf{P^{-1}_{\rm DG}A}_{\rm DG})$ & $\frac{(\gamma_{\rm max})_{N_{\rm dof}}}
{(\gamma_{\rm min})_{1}}$ & it PCG\\
\hline
$10$ &  & 16.2 &  31.9 & 24  & & 2.0 & 2.0 &  5  \\
$20$ &  & 19.2 &  31,9 &  25 & &2,0  & 2.0 &  5  \\
$30$ &  & 20.2 & 31.9  &  30 & & 2.0 & 2.0 &  5  \\
$40$ &  & 20.8 &  31.9 &  30 & & 2.0 & 2.0 &  5  
\end{tabular}
}

and for $c_\sigma=20$:\\

\centerline{\begin{tabular}{c|cccc|cccc}
$N_1=N_2$&  $a_{\rm p,1}$& $\kappa(\mathsf{P^{-1}_{\rm DG}A}_{\rm DG})$ & $\frac{(\gamma_{\rm max})_{N_{\rm dof}}}
{(\gamma_{\rm min})_{1}}$ & it PCG 
&  $a_{\rm p,2}$&$\kappa(\mathsf{P^{-1}_{\rm DG}A}_{\rm DG})$ & $\frac{(\gamma_{\rm max})_{N_{\rm dof}}}
{(\gamma_{\rm min})_{1}}$ & it PCG\\
\hline
$10$ &  & 14.4 & 18.7  &  21   & & 2.0 & 2.0  &  5  \\
$20$ &  & 16.5 & 18.7  &  23   & & 2.0 &  2.0 &  5  \\
$30$ &  & 17.1 & 18.7  &  27   & & 2.0 & 2.0  &  5  \\
$40$ &  & 17.4 &  18.7 &  27   & & 2.0 & 2.0  & 5   
\end{tabular}
}

respectively.  
We can see that for both unpreconditioned Galerkin and DG methods the numbers of CG steps increase with an increasing problem size.
In contrast, the condition numbers and the 
numbers of CG steps stagnate for the 
preconditioned problems.
Moreover, the condition numbers are reliably and quite accurately
estimated by the ratios $\frac{(\gamma_{\rm max})_{N_{\rm dof}}}
{(\gamma_{\rm min})_{1}}$.
We can also notice that a higher penalty coefficient leads to
the larger condition number and number of CG steps for DG method
but not for the preconditioned DG method.
And, of course, the closer the reference data are to the 
original problem, the better preconditioner is obtained.
}\end{ex}

\begin{ex}\label{example2}{\rm
In this example we show approximation quality of the guaranteed
two-sided bounds to the eigenvalues
of the preconditioned problems. For the Galerkin discretization
(with continuous basis functions), 
the matrices $\mathsf{A}_k$ and $\mathsf{P}_k$ correspond to the element integrals (see~\cite{LadeckyPZ2020,PuLa2021}), while for the DG discretization they correspond to the edge
integrals plus one thirds of the attached elements'
integrals, as defined in~\eqref{A_terms} and~\eqref{P_terms}.
We consider the same setting as in Example~\ref{example1}, $N_1=N_2=10$, $c_\sigma=2$, $f=10$. We present the 
sorted eigenvalues and their lower and upper bounds 
$\gamma_{\rm min}$ and $\gamma_{\rm max}$, respectively,
for three different test problems, where $a,a_{\rm p}:\Omega\to\mathbb{R}^{2\times 2}$ and $c,c_{\rm p}:\Omega\to\mathbb{R}$
are defined in the following table:\\

\centerline{\begin{tabular}{c|cccc}
 & $a$ & $a_{\rm p}$  &  $c$  & $c_{\rm p}$\\
  \hline
 Test~1 & $(1+\sin(x_1x_2\pi))\,{\rm diag}(3,1)+\frac{1}{10} I$  &  ${\rm diag}(3,1)$  &  1  & 1\\
 Test~2  &  $(1+\sin(x_1x_2\pi))\,{\rm diag}(3,1)+\frac{1}{10} I$ &  ${\rm diag}(3,1)$  &  0  & 0\\
 Test~3  &  $I$ ($x_1<0.5$),\; $5I$ \; ($x_1\ge 0.5$)  &  $I$   &  0  & 0
\end{tabular}
}

The resulting spectra and their bounds 
$\gamma_{\rm min},\gamma_{\rm max}\in\mathbb{R}^{N_{\rm dof}}$
are graphically displayed in Figure~\ref{fig1}.
In the DG method, the bounds are obtained from Algorithm~1.
For the Galerkin method the bounds are obtained analogically,
see e.g.~\cite{LadeckyPZ2020,PuLa2021}.

\begin{figure}[h]
\centering
\includegraphics[width=\textwidth]{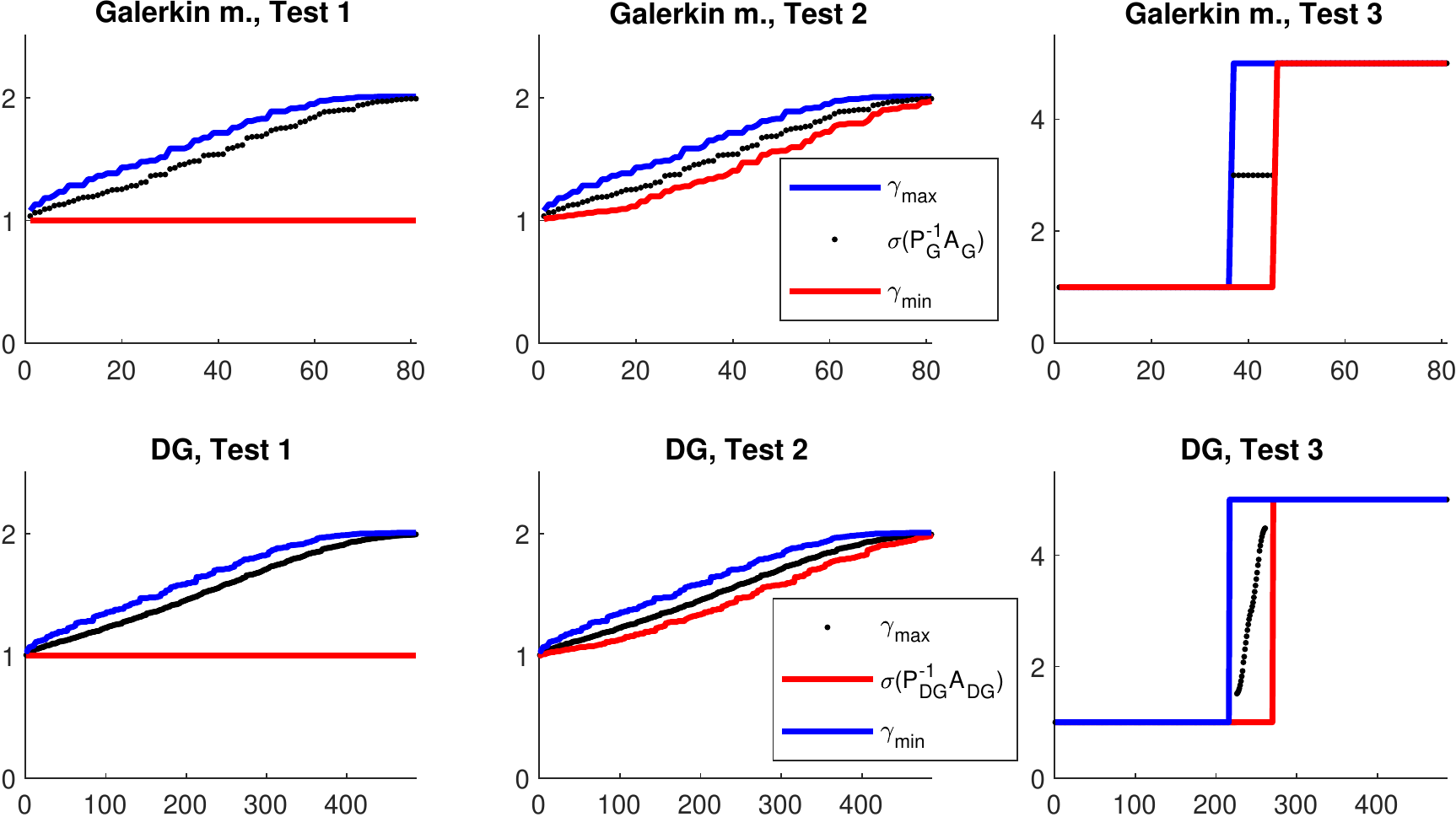}   
    \caption{Spectra of $\mathsf{P}^{-1}\mathsf{A}$ (black dots)
    and their upper (blue lines) and
    lower (red lines) bounds for Galerkin (first line of figures) and DG discretization (second line of figures) for test problems 
    (Tests~1, 2, 3) defined in Example~\ref{example2}.}
   \label{fig1}
\end{figure}
}\end{ex}

\section{Preconditioned Galerkin method and convection-\\
diffusion-reaction problem}\label{secG}

In this section, as an example of a non-symmetric problem,
we consider a scalar convection-diffusion-reaction
equation~\cite{DolejsiFeistauer2015}. The model problem reads 
\begin{equation}\label{model2}
-\nabla\cdot\left(a(x)\nabla u(x)\right)+b(x)\cdot \nabla u(x)+c(x)u(x)=f(x),\;\; x\in \Omega
\end{equation}
where $u(x)=0$ for $x\in\partial\Omega$.
The coefficient
$a:\Omega \to {\mathbb R}^{2\times 2}$ is essentially bounded
and uniformly s.p.d.~in $\Omega$,
$c$ is nonnegative and essentially bounded in $\Omega$, 
$b\in W^{1,\infty}(\Omega)^2$, and
$f\in L^2(\Omega)$.
We assume $c-\frac{1}{2}\nabla\cdot b\ge 0$ in $\Omega$ that guarantees the ellipticity of the problem (\ref{model2}).
We apply the same notation and assumptions on the mesh as in Section~\ref{secG}.
Moreover, we assume for simplicity 
that the problem (\ref{model2}) is not convection-dominated;
i.e., we assume that the P\' eclet number ${\rm Pe}=\frac{\Vert b\Vert h_\tau}{\Vert a\Vert}$ is sufficiently small, where $h_\tau=\mbox{diam}(\tau)$. Otherwise, the discretization of the convective term needs to be stabilized. For the overview of dicretization methods for the convection dominated problems and the corresponding stabilizations, see e.g.~\cite{RST}. Although the stabilization 
would not complicate our preconditioning method,
we do not consider it here in order to focus on the explanation of the most
important ideas.
For numerical solution let us define the space of 
continuous and element-wise linear functions $V_h$ satisfying the boundary condition.
Then the discretized weak form of the equation \eqref{model2} reads to find $u\in V_h$ such that 
\begin{equation}\nonumber
\mathcal{A}(u,v)_h=(f,v),\quad v\in V_h,
\end{equation}
where
\begin{eqnarray}
\mathcal{A}(u,v)_h &=&A(u,v)_h+B(u,v)_h+C(u,v)_h,
\nonumber
\end{eqnarray}
where
\begin{equation}\nonumber
A(u,v)_h=\int_{\Omega}a\nabla u\cdot\nabla v\,
{\rm d}x,\quad B(u,v)_h=\int_\Omega
b\cdot\nabla uv\,{\rm d}x,\quad
C(u,v)_h=\int_{\Omega}c\, u v\,
{\rm d}x.
\end{equation}
Let us assume that the FE basis functions are defined by
$N_{\rm dof}$ DOFs.
The system of linear equations reads
\begin{equation}\nonumber
\mathsf{(A+B)u}=\mathsf{f}
\end{equation}
where $\mathsf{A,B}\in \mathbb{R}^{N_{\rm dof}\times N_{\rm dof}}$
and $\mathsf{f}\in \mathbb{R}^{N_{\rm dof}}$.
Matrices $\mathsf{A}$ and $\mathsf{B}$ are 
the symmetric and skew-symmetric parts of the system matrix, respectively.
Denoting the FE basis functions by $\varphi_j$,
$j=1,\dots,N_{\rm dof}$, and from
$u=\sum_{j=1}^{N_{\rm dof}}\mathsf{u}_j\varphi_j$, we have 
\begin{equation}\nonumber
\mathsf{A}_{ij}+\mathsf{B}_{ij}=\mathcal{A}(\varphi_j,\varphi_i)_h,
\quad \mathsf{f}_i=(f,\varphi_i), \quad i,j=1,\dots, N_{\rm dof}.
\end{equation}
The s.p.d.~preconditioning matrix $\mathsf P$ is obtained in the same way as $\mathsf A$ but for the reference data
$a_{\rm p}$, $c_{\rm p}$ and $b_{\rm p}=(0,0)^T$.
The matrices $\mathsf{A}_k$, $\mathsf{B}_k$ and $\mathsf{P}_k$ 
from Section~\ref{NonSymMat} are here associated with an element $\tau_k$.
Therefore, the notation from Section~\ref{NonSymMat}
matches to the notation of this section
when $N=N_{\rm dof}$ and $N_{\rm loc}=N_\tau$.
For $k=1,\dots,N_\tau$ and $i,j=1,\dots,N_{\rm dof}$ we obtain
\begin{eqnarray}
    (\mathsf{A}_k)_{ij}&=&\int_{\tau_k}
    a\nabla \varphi_j\cdot\nabla\varphi_i\,\dd x+
      \int_{\tau_k}    c\, \varphi_j\varphi_i\,\dd x +
       \frac{1}{2}\int_{\tau_k}    b\cdot \nabla\varphi_j\varphi_i
    +b\cdot \nabla\varphi_i\varphi_j\,\dd x
             \label{AN_terms}\label{AN_terms}\\
        (\mathsf{B}_k)_{ij}&=&
    \frac{1}{2}\int_{\tau_k}    b\cdot \nabla\varphi_j\varphi_i
    -b\cdot \nabla\varphi_i\varphi_j\,\dd x \label{BN_terms}
\end{eqnarray}
and
\begin{equation}\label{PN_terms}
    (\mathsf{P}_k)_{ij}=\int_{\tau_k}
    a_{\rm p}\nabla \varphi_j\cdot\nabla\varphi_i\,\dd x+     
        \int_{\tau_k}    c_{\rm p}\, \varphi_j\varphi_i\,\dd x.
\end{equation}
We might notice that the contribution of the convective term into the system matrix will be pure skew-symmetric if $\nabla \cdot b=0$. However,
its individual element-wise contributions can have non-zero
symmetric part which comes from
element boundary integrals. Therefore, when composing the
system matrix split into the symmetric and
skew-symmetric parts $\mathsf{A}$ and $\mathsf{B}$, respectively,
we can take this into account and ignore the contribution coming from the convective term in composition of $\mathsf{A}$ 
if $\nabla \cdot b=0$.

Now, we can introduce the algorithm for composing the matrices $\mathsf{A,B,P}$ and bounds to the spectrum 
of $\mathsf{P}^{-1}(\mathsf{A+B})$ according to Lemma~\ref{lemNonsym1}.

{\bf Algorithm~2.} (Constructing $\mathsf{A},\mathsf{B},\mathsf{P}\in\mathbb{R}^{N_{\rm dof}\times N_{\rm dof}}$ and 
spectral bounds $\alpha_{\rm min},\alpha_{\rm max},\beta_{\rm max}\in\mathbb{R}$.)
\begin{itemize}
    \item[1.]  
    Set $\alpha_{\rm min}:=\infty$, $\alpha_{\rm max}=0$, $\beta_{\rm max}=0$.
    \item[2.]
    For every element $\tau_k$, $k=1,\dots,N_\tau$: \\
    Set $\widetilde{\mathsf{A}}_k,\widetilde{\mathsf{B}}_k,\widetilde{\mathsf{P}}_k \in 
    \mathbb{R}^{3\times 3}$ according to~\eqref{AN_terms}, \eqref{BN_terms}, and~\eqref{PN_terms}, respectively.\\
    Add matrices $\widetilde{\mathsf{A}}_k,\widetilde{\mathsf{B}}_k,
    \widetilde{\mathsf{P}}_k$ to elements at positions given by indices $j_1,j_2,j_3$ in $\mathsf{A,B}$, and $\mathsf{P}$, respectively.\\
    Compute generalized eigenvalues of
    $\widetilde{\mathsf{A}}_k\mathsf{v}=\lambda\widetilde{\mathsf{P}}_k\mathsf{v}$ and 
    $\widetilde{\mathsf{B}}_k\mathsf{v}=\lambda\widetilde{\mathsf{P}}_k\mathsf{v}$    and set \\
    $\alpha_1:=\lambda_{\rm ker,min}(\widetilde{\mathsf{P}}_k^{-1}\widetilde{\mathsf{A}}_k)  $,\\
    $\alpha_3:=\lambda_{\rm ker,max}(\widetilde{\mathsf{P}}_k^{-1}\widetilde{\mathsf{A}}_k)  $,\\
    $\beta_3:=\lambda_{\rm ker,im,max}(\widetilde{\mathsf{P}}_k^{-1}\widetilde{\mathsf{B}}_k)  $.\\  
    For $s=1,2,3$: \\
    $\alpha_{\rm min}:=\min\{\alpha_1,\alpha_{\rm min} \}$\\
    $\alpha_{\rm max}:=\max\{\alpha_3,\alpha_{\rm max}\}$\\
    $\beta_{\rm max}:=\max\{\beta_3, \beta_{\rm max}\}$. 
\end{itemize}

In the following numerical example we consider 
a convection-diffusion-reaction problem preconditioned 
with an s.p.d.~matrix.

\begin{ex}{\rm
We consider equation~\eqref{model2} and its preconditioning
problem with $\Omega=(0,1)\times (0,1)$,
\begin{equation}\nonumber
    a(x) = \left(\begin{array}{cc}
    20-2x_2&0\\
    0&3-2x_1
    \end{array}\right),\qquad
     a_{\rm p,1}(x) = \left(\begin{array}{cc}
    1&0\\    0&1
    \end{array}\right),
    \qquad
     a_{\rm p,2}(x) = \left(\begin{array}{cc}
    19&0\\    0&2
    \end{array}\right),
\end{equation}
$c=10$, $c_{\rm p}=10$, $b=10\, (-x_2,x_1)^T$, $b_{\rm p}=(0,0)^T$, $f=10$.
We obtain systems of linear equations with 
symmetric and skew-symmetric parts $\mathsf{A}$ and 
$\mathsf{B}$, respectively, and the s.p.d.~preconditioning matrix 
$\mathsf{P}$. We obtain the following condition numbers and 
steps of (preconditioned) GMRES with no restart~\cite{
LiesenS,SaadBook,vanderVorst}
with the tolerance for relative residual norms $10^{-8}$.
The spectral bounds $\alpha_{\rm min}$ and $\alpha_{\rm max}$
and $\beta_{\rm max}$ to the eigenvalues of $\mathsf{P}^{-1}\mathsf{A}$ and  $\mathsf{P}^{-1}\mathsf{B}$, respectively,
are obtained from Algorithm~2. In this example,
matrix $\mathsf P$ and $\mathsf A$ are s.p.d., thus we can also 
obtain the upper bounds $\frac{\alpha_{\rm max}}{\alpha_{\rm min}}$ to the condition numbers of 
$\mathsf{P}^{-1}\mathsf{A}$.
First we present the results for the unpreconditioned problem:\\

\centerline{\begin{tabular}{c|ccc}
$N_1=N_2$& $\kappa(\mathsf{A})$& $\lambda_{\rm im,max}(\mathsf{B})$  &it GMRES \\
\hline
10&  3.9e1&  3.8 &  44 \\
30& 3.7e2 &  1.6 &  137 \\
50& 1.0e3 &  1.0 &  232 \\
70& 2.0e3 & 0.8  &   329
 \end{tabular}
 }

Then we show the results for the preconditioner with $a_{\rm p,1}$: \\

\centerline{\begin{tabular}{c|ccccc}
$N_1=N_2$&  $\kappa(\mathsf{P^{-1}A})$ 
& $\frac{\alpha_{\rm max}}{\alpha_{\rm min}}$
& $\lambda_{\rm im,max}(\mathsf{P}^{-1}\mathsf{B})$ &
$\beta_{\rm max}$ & it pGMRES  \\
\hline
10  &  8.0 &  19.8 & 2.1 & 6.4 & 25  \\
30  & 11.9  & 20.0  & 2.2  & 6.6& 31\\
50  &  13.6 & 20.0  &  2.2 & 6.6 & 32\\
70  & 14.5  &  20.0 & 2.2  &6.7 & 33
 \end{tabular}
 }

and for the preconditioner with $a_{\rm p,2}$:\\

\centerline{\begin{tabular}{c|ccccc}
$N_1=N_2$&  $\kappa(\mathsf{P^{-1}A})$ 
& $\frac{\alpha_{\rm max}}{\alpha_{\rm min}}$
& $\lambda_{\rm im,max}(\mathsf{P}^{-1}\mathsf{B})$ &
$\beta_{\rm max}$ & it pGMRES    \\
\hline
10&  1.4   & 2.7 & 0.36 & 3.4  &  11  \\
30&  1.8  & 2.9  &  0.40 & 3.5 &  13 \\
50&   2.1 &  2.9 & 0.41 & 3.5 &  14 \\
70&  2.2  & 3.0  & 0.41 &  3.5 &   14
 \end{tabular}
 }
In this example, we can notice that while for the refining mesh
the condition numbers of the real parts of the systems matrices grow
quickly, the condition numbers of $\mathsf{P}^{-1}\mathsf{A}$
stagnate and the extreme eigenvalue of $\mathsf{P}^{-1}\mathsf{B}$ stays bounded. The same is true for the numbers of steps of the 
GMRES method.
Again, the preconditioner is more efficient when the reference data 
are closer to the data of the underlying problem.
}\end{ex}

\section{Conclusion}

The main goal of this paper is to introduce 
a recently developed type of preconditioning to DG discretization method and to convection-diffusion-reaction
problems which can lead to non-symmetric matrices of the resulting systems
of linear equations. Numerical solution is considered to be calculated
using FEM, but many other discretization methods can be used.
While the DG method is often used for convection-diffusion problems such as Navier-Stokes equations, we consider these two special 
cases (DG and non-symmetry) separately, in order to illuminate
their particular features in an accessible way.
The basic idea of the preconditioning lies 
in construction a preconditioning matrix which is obtained in a similar manner as the original matrix but for different (reference) data. The construction also yields spectral bounds
for the resulting matrix.

It is important to note that the preconditioning matrix 
is in fact a matrix of some problem of a similar 
complexity as the original one.
Therefore, using it as a preconditioner may be
efficient only in some special cases.
For example, if the domain and the mesh have a regular structure, such as pixel or voxel based (\cite{LadeckyLFPPJZ2023,LeuteLFJPZJP2022,ThieleR2022}),
then inverting the matrix of the preconditioning problem with constant data can be extremely cheap using the fast discrete Fourier
transform. If one needs to solve many problems of a similar nature,
it can be worth to factorize the preconditioning matrix once and then
use it repeatedly. 
Therefore, we do not present the solution time in our examples,
because it strongly depends on the problem and on the implementation.
As an inherent part of the method, guaranteed lower and upper bounds to every eigenvalue
of the resulting matrix (in case of symmetric matrices) 
or at least to the spectrum as a whole (in case of non-symmetric
matrices) are obtained. Having an information about the eigenvalue distribution
can help to evaluate the effectivity of the preconditioner 
in advance; see e.g.~\cite{G}.
Moreover, having a good preconditioner 
with guaranteed tight spectral bounds allows calculating a guaranteed 
energy norm of the error just from a current residual.

As an open question we consider a possibility to obtain 
even more accurate bounds to complex eigenvalues of preconditioned sum
of symmetric and skew-symmetric matrices than provided by 
Lemma~\ref{lemNonsym1}. We hope for some analogy to the bounds 
mentioned in Remark~\ref{remG}, which is backed by numerical experiments. However,
the counter-example in Remark~\ref{rem_counter} shows that 
some modification of the statement would be necessary.

{\bf Acknowledgement.}
LG, ML, IP, and MV acknowledge funding by the European Regional Development Fund (Centre of Advanced Applied Sciences -- CAAS, CZ 02.1.01/0.0/0.0/16\_019/0000778), by  the Czech Science Foundation (projects No.~19-26143X (JZ), No.~20-14736S (ML, MV), and No.~22-35755K (LG)), 
the Student Grant Competition of the Czech Technical University in Prague
(project No.~SGS23/002/OHK1/1T/11 (LG)).
None of the authors have a conflict of interest.

\end{document}